\documentclass[preprint,EJP]{ejpecp}

\usepackage[T1]{fontenc}
\usepackage[utf8]{inputenc}

\usepackage{stmaryrd}

\SHORTTITLE{Overdamped limit of the Langevin quasi-stationary distribution}

\TITLE{Quasi-stationary distribution for the {L}angevin process in cylindrical domains, part II: overdamped limit\thanks{Supported
    by the Région Ile-de- France through a PhD fellowship of the Domaine d'Intérêt Majeur (DIM) Math Innov.}}  
 
\AUTHORS{%
  Mouad Ramil\footnote{Cermics (Ecole des Ponts ParisTech), France.
    \EMAIL{ramil.mouad@gmail.com}}}

\KEYWORDS{Langevin process ; overdamped limit ; quasi-stationary distribution ; overdamped Langevin process} 

\AMSSUBJ{82C31 ; 35B25 ; 47B07 ; 60H10}  

\SUBMITTED{July 17, 2021} 
\ACCEPTED{April 30, 2022}

\ARXIVID{2103.00338}  
\HALID{hal-03161373}  

\VOLUME{0}
\YEAR{2020}
\PAPERNUM{0}
\DOI{10.1214/YY-TN} 

\ABSTRACT{Consider the Langevin process, described by a vector (positions and momenta) in $\mathbb{R}^{d}\times\mathbb{R}^d$. Let~$\mathcal O$ be a $\mathcal{C}^2$ open bounded and connected set of $\mathbb{R}^d$. Recent works showed the existence of a unique quasi-stationary distribution (QSD) of the Langevin process on the domain $D:=\mathcal{O}\times\mathbb{R}^d$. In this article, we study the overdamped limit of this QSD, i.e. when the friction coefficient goes to infinity. In particular, we show that the marginal law in position of the overdamped limit is the QSD of the overdamped Langevin process on the domain $\mathcal{O}$.}

\begin{document}



\section{Introduction} 
In statistical physics, the evolution of a molecular system at a given temperature is
typically modeled by the Langevin dynamics 
\begin{equation}\label{eq:Langevin_intro qsd}
  \left\{
    \begin{aligned}
        &\mathrm{d}q_t=M^{-1} p_t \mathrm{d}t , \\
        &\mathrm{d}p_t=F(q_t) \mathrm{d}t -\gamma M^{-1}  p_t
        \mathrm{d}t +\sqrt{2\gamma\beta^{-1}} \mathrm{d}B_t ,
    \end{aligned}
\right.  
\end{equation}
where $d=3N$ for a number $N$ of particles, $(q_t,p_t) \in \mathbb{R}^d \times \mathbb{R}^d$ denotes the set of positions and momenta of the
particles, $M \in \mathbb{R}^{d \times d}$ is the mass matrix,
$F:\mathbb{R}^d \to \mathbb{R}^d$ is the force acting on the particles, $\gamma >0$ is
the friction parameter, and $\beta^{-1}= k_B T$ with $k_B$
the Boltzmann constant and $T$
the temperature of the system. Alternatively, the overdamped Langevin dynamics
\begin{equation}\label{eq:ovLangevin_intro}
  \mathrm{d}\overline{q}_t = F(\overline{q}_t)\mathrm{d} t + \sqrt{2\beta^{-1}}\mathrm{d} B_t,
\end{equation}
may also be employed. Notice that both processes are related by the fact that when the force field $F$ is conservative, that is to say when there exists $V : \mathbb{R}^d \to \mathbb{R}$ such that $F = -\nabla V$, then the stationary distribution of $(\overline{q}_t)_{t \geq 0}$ writes
\begin{equation}\label{eq:nuovLangevin}
  \overline{\nu}(\mathrm{d}q) = \frac{1}{Z}\mathrm{e}^{-\beta V(q)}, \qquad Z = \int_{\mathbb{R}^d} \mathrm{e}^{-\beta V(q)} \mathrm{d}q,
\end{equation}
while the stationary distribution of $(q_t,p_t)_{t \geq 0}$ has the product structure
\begin{equation}\label{eq:nuLangevin}
  \nu(\mathrm{d}q\mathrm{d}p) = \overline{\nu}(\mathrm{d}q) \frac{\mathrm{e}^{-\frac{\beta\vert p\vert^2}{2}}}{(2\pi\beta^{-1})^{\frac{d}{2}}}\mathrm{d}p.
\end{equation} 

The dynamics presented above are used in particular to
compute thermodynamic and dynamic quantities, with numerous
applications in biology, chemistry and materials science. In many practical situations of interest, the system remains trapped for very long times
in subsets of the phase space, called metastable states, see for
example~\cite[Sections 6.3 and 6.4]{LelSto16}. Typically, these states
are defined in terms of positions only, and are thus open sets $\mathcal{O}$ of $\mathbb{R}^d$ for~\eqref{eq:ovLangevin_intro} or  cylinders of the
form $D=\mathcal{O} \times \mathbb{R}^d$ for~\eqref{eq:Langevin_intro qsd}. In such a case, it is expected that the process reaches a local equilibrium distribution within the metastable state before leaving it. This distribution is called the quasi-stationary distribution (QSD). The existence of this limiting behavior has been shown recently in~\cite{LelRamRey2}, using compactness arguments. Similar results can also be found in the recent works:~\cite[Chapter 4]{RamPHD} based on criterias developed in~\cite{V} by N. Champagnat and D. Villemonais and in~\cite{GuiNectoux} using a Lyapunov and an Harnack inequality argument.

The motivation for this work comes from the well-known fact that, when $\gamma \to \infty$, for all $T>0$, the process $(q_{\gamma t})_{t\in[0,T]}$ converges in distribution to $(\overline{q}_{t})_{t\in[0,T]}$, hence the name overdamped Langevin dynamics for~\eqref{eq:ovLangevin_intro} (see~\cite[Proposition 2.15]{LelRouSto10} and~\cite{Kramers,Freidlin} for instance) on the space of continuous functions on $[0,T]$, endowed with the supremum norm on $[0,T]$. Therefore, it is expected that the marginal law in position of the QSD on $D$ of $(q_t,p_t)_{t\geq0}$ converges weakly to the QSD on $\mathcal{O}$ of the overdamped Langevin process. We actually prove a more general result by perturbing the Langevin dynamics to obtain an independent couple, which will allow us to consider the marginals separately, making the proof much easier. To the best of our knowledge, this result is the first to provide an overdamped limit of the couple (position,velocity) for the Langevin process. We are then able to identify the weak limit of the QSD on $D$, from which we can easily deduce the weak convergence of the marginal distributions.  

More precisely, we study the limit of the QSD on $D$, of $(q_t,p_t)_{t\geq0}$, when the friction parameter $\gamma$ goes to infinity and show that it converges to the product measure
\begin{equation}\label{product structure mu infini}
  \mu^{(\infty)}(\mathrm{d}q\mathrm{d}p) = \overline{\mu}(\mathrm{d}q) \frac{\mathrm{e}^{-\frac{\beta\vert p\vert^2}{2}}}{(2\pi\beta^{-1})^{\frac{d}{2}}}\mathrm{d}p,
\end{equation}
where $\overline{\mu}$ is the QSD of the overdamped Langevin process $(\overline{q}_t)_{t \geq 0}$ in $\mathcal{O}$. This result is stated in Section~\ref{sec: main results qsd} and it relies on recent results on the Langevin process which are recalled in Section~\ref{tightness convergence langevin}.
 
\section{Main results}\label{sec: main results qsd}  

We first introduce in Section~\ref{sec:OL} the notion of quasi-stationary distribution (QSD) and recall well-known results for the QSD of the overdamped Langevin process on a smooth bounded domain $\mathcal{O}$. We also recall in Section~\ref{sec:Lang} recent results from the companion paper~\cite{LelRamRey2} related to the existence of a unique QSD of the Langevin process on the domain $D:=\mathcal{O}\times\mathbb{R}^d$. Finally, we state our main result regarding the overdamped limit of the Langevin QSD on $D$ in Section~\ref{Main results_overdamped}.

\subsection{Quasi-stationary distribution for the overdamped Langevin process} \label{sec:OL}
 The notion of quasi-stationary distribution (QSD) is central in this text. We recall here its definition in a general setting, and refer to~\cite{Collet,VilMel} for a complete introduction. 

Let $E$ be a Polish space endowed with its Borel $\sigma$-algebra $\mathcal{B}(E)$, and let $(X_t)_{t \geq 0}$ be a time-homogeneous, strong Markov process in $E$ with continuous sample-paths. For any $x \in E$, we denote by $\mathbb{P}_x$ the probability measure under which $X_0=x$ almost surely, and for any probability measure $\theta$ on $E$, we define
\begin{equation*}
  \mathbb{P}_\theta(\cdot) := \int_E \mathbb{P}_x(\cdot)\theta(\mathrm{d}x).
\end{equation*}
Let $D$ be an open subset of $E$ and $\tau_\partial$ be the stopping time defined by
\begin{equation*}
  \tau_\partial := \inf\{t>0: X_t \not\in D\}.
\end{equation*}

\begin{definition}[QSD]\label{def QSD 1} A probability measure $\mu$ on $D$ is said to be a QSD on $D$ of the process $(X_t)_{t\geq0}$, if for all $A\in\mathcal{B}(D) := \{A\cap D,A\in\mathcal{B}(E)\}$, for all $t\geq0$,
\begin{equation}\label{eq:defqsd}
  \mathbb{P}_\mu(X_t\in A, \tau_\partial>t)=\mu(A)\mathbb{P}_\mu(\tau_\partial>t).
\end{equation}
\end{definition} 
\noindent When $\mathbb{P}_\mu(\tau_\partial>t)>0$, the identity~\eqref{eq:defqsd} equivalently writes $\mathbb{P}_\mu(X_t\in A|\tau_\partial>t)=\mu(A)$. Now let $\beta>0$ and $F:\mathbb{R}^d\mapsto\mathbb{R}^d$ satisfying the following assumption. 
\begin{hypothesis}\label{hyp F1 qsd}
$F\in\mathcal{C}^\infty(\mathbb{R}^{d},\mathbb{R}^{d})$. \end{hypothesis}

Let $(\Omega,\mathcal{F},(\mathcal{F}_t)_{t\geq0},\mathbb{P})$ be a filtered probability space and $(B_t)_{t\geq0}$ a $d$-dimensional $(\mathcal{F}_t)_{t\geq0}$-Brownian motion. Under Hypothesis~\ref{hyp F1 qsd}, the vector field $F$ is locally Lipschitz continuous and therefore the stochastic differential equation~\eqref{eq:ovLangevin_intro} possesses a unique strong solution $(\overline{q}_t)_{0 \leq t < \overline{\tau}_\infty}$ defined up to some explosion time $\overline{\tau}_\infty \in (0,+\infty]$.
Let $\mathcal{O}$ be an open set of $\mathbb{R}^d$ satisfying the following assumption.
\begin{hypothesis}\label{hyp O qsd}
$\mathcal{O}$ is an open $\mathcal{C}^2$ bounded connected set of $\mathbb{R}^d$. 
\end{hypothesis}
Let $\overline{\tau}_{\partial}:=\inf \{t>0: \overline{q}_t\notin \mathcal{O}\}$ be the first exit time from $\mathcal{O}$ of the process $(\overline{q}_t)_{0 \leq t < \overline{\tau}_\infty}$. Under Hypotheses~\ref{hyp F1 qsd} and~\ref{hyp O qsd}, the vector field $F$ is Lipschitz continuous on $\mathcal{O}$ and therefore $\overline{\tau}_\partial \leq \overline{\tau}_\infty$. 

It has been shown in~\cite{V3,GQZ,LebLelPer,KnobPart} that the overdamped Langevin process admits a unique QSD on $\mathcal{O}$, which moreover satisfies the following properties.
\begin{theorem}[QSD of the overdamped Langevin process]\label{qsd overdamped} Under Hypotheses~\ref{hyp F1 qsd} and~\ref{hyp O qsd}, there exists a unique QSD $\overline{\mu}$ on $\mathcal{O}$ of the process $(\overline{q}_t)_{t\geq0}$. Furthermore, 
\begin{enumerate} 
\item there exists $\overline{\psi}\in\mathcal{C}^2(\mathcal{O})\cap\mathcal{C}^b(\overline{\mathcal{O}})$ such that $\overline{\mu}(\mathrm{d}q)=\overline{\psi}(q)\mathrm{d}q$, where $\mathrm{d}q$ is the Lebesgue measure on~$\mathbb{R}^d$,
\item there exists $\overline{\lambda}_0>0$ such that, if $\overline{q}_0$ is distributed according to $\overline{\mu}$, then $\overline{\tau}_\partial$ follows the exponential law with parameter $\overline{\lambda}_0$.
\end{enumerate}
\end{theorem}
\subsection{Quasi-stationary distribution for the Langevin process}\label{sec:Lang}
In this section we recall some results from~\cite{LelRamRey2} that will be used henceforth. Let $\gamma>0$ and $\beta>0$ independent of $\gamma$. Consider now the following Langevin process
\begin{equation}\label{eq:Langevin qsd gamma}
  \left\{
    \begin{array}{ll}
        \mathrm{d}q^{(\gamma)}_t=p^{(\gamma)}_t \mathrm{d}t , \\
        \mathrm{d}p^{(\gamma)}_t=F(q^{(\gamma)}_t) \mathrm{d}t-\gamma  p^{(\gamma)}_t \mathrm{d}t+\sqrt{2\gamma\beta^{-1}}\mathrm{d}B_t,
    \end{array}
\right.  
\end{equation}
Under Hypothesis~\ref{hyp F1 qsd}, the stochastic differential equation~\eqref{eq:Langevin qsd gamma} possesses a unique strong solution $(X^{(\gamma)}_t=(q^{(\gamma)}_t,p^{(\gamma)}_t))_{0 \leq t < \tau^{(\gamma)}_\infty}$, defined up to some explosion time $\tau^{(\gamma)}_\infty \in (0,+\infty]$. Notice that, compared to~\eqref{eq:Langevin_intro qsd}, we
consider here and
  henceforth the mass to be identity, so that momentum is identified with velocity. 
 
Let $\tau^{(\gamma)}_{\partial}$ be the first exit time from $D$ of the Langevin process $(X^{(\gamma)}_t)_{t\geq0}$ in~\eqref{eq:Langevin qsd gamma}, i.e.  $$\tau^{(\gamma)}_{\partial}=\inf \{t>0: X^{(\gamma)}_t\notin D\} .$$
Under Hypotheses~\ref{hyp F1 qsd} and~\ref{hyp O qsd}, $F$ is Lipschitz continuous on $\mathcal{O}$ and therefore $\tau^{(\gamma)}_\partial \leq \tau^{(\gamma)}_\infty$. 
 Concerning the existence of a QSD on the domain $D:=\mathcal{O}\times\mathbb{R}^d$, similar proofs, as in the overdamped Langevin case, do not apply here. In fact, the infinitesimal generator of the process $(X^{(\gamma)}_t)_{t\geq0}$ is not elliptic but only hypoelliptic, and the natural domain $D = \mathcal{O} \times \mathbb{R}^d$ is not bounded, even if $\mathcal{O}$ is bounded. However, using a compactness argument, analogous results to Theorem~\ref{qsd overdamped} for the Langevin process~\eqref{eq:Langevin qsd gamma} are obtained in~\cite{LelRamRey2}:  
\begin{theorem}[QSD of the Langevin process]\label{thm:qsd langevin} Under Hypotheses~\ref{hyp F1 qsd} and~\ref{hyp O qsd}, there exists a unique QSD $\mu^{(\gamma)}$ on $D$ of the process $(X^{(\gamma)}_t)_{t\geq0}$. Furthermore, 
\begin{enumerate} 
\item there exists $\psi^{(\gamma)}\in\mathcal{C}^2(D)\cap\mathcal{C}^b(\overline{D})$ such that $\mu^{(\gamma)}(\mathrm{d}q\mathrm{d}p)=\psi^{(\gamma)}(q,p)\mathrm{d}q\mathrm{d}p$, where $\mathrm{d}q\mathrm{d}p$ is the Lebesgue measure on~$\mathbb{R}^{2d}$,
\item there exists $\lambda^{(\gamma)}_0>0$ such that, if $X^{(\gamma)}_0$ is distributed according to $\mu^{(\gamma)}$, then $\tau^{(\gamma)}_\partial$ follows the exponential law with parameter $\lambda^{(\gamma)}_0$.
\end{enumerate}
\end{theorem}

\subsection{Main result: Overdamped limit of the Quasi-stationary distribution of the Langevin process} \label{Main results_overdamped}  

To state the main results of this work, it is more convenient to keep track of the initial value $q$ (resp. $x=(q,p)$) of the solution to~\eqref{eq:ovLangevin_intro} (resp. to~\eqref{eq:Langevin qsd gamma}) by denoting the latter by $(\overline{q}^q_t)_{t \geq 0}$ (resp. $(X^{(\gamma),x}_t = (q^{(\gamma),x}_t,p^{(\gamma),x}_t))_{t \geq 0}$). Moreover, we need the following strengthening of Hypothesis~\ref{hyp F1 qsd}.
\begin{hypothesis}\label{hyp F2 qsd}
$F\in\mathcal{C}^\infty(\mathbb{R}^{d},\mathbb{R}^{d})$ and $F$ is bounded and globally Lipschitz continuous on $\mathbb{R}^d$. 
\end{hypothesis} 

The following theorem will be the key to obtain the overdamped limit of the QSD. It is an extension of the well-known convergence of the position marginal $(q_{\gamma t})_{t\in[0,T]}$ for the Langevin to the couple $((q_{\gamma t})_{t\in[0,T]},p_{\gamma T})$ using a novel perturbative argument.
\begin{theorem}[Generalization of the overdamped limit of the Langevin process]\label{cv loi indep} Assume that Assumption~\ref{hyp F2 qsd} holds.
Let $T>0$ and $x=(q,p)\in\mathbb{R}^{2d}$. Let $Z  \sim  \mathcal{N}_{d}(0,\beta^{-1} I_d)$ be a Gaussian vector independent of the process $(\overline{q}^q_t)_{t\in[0,T]}$. The law of the couple $((q^{(\gamma),x}_{\gamma t})_{t\in[0,T]},p^{(\gamma),x}_{\gamma T})$ converges weakly to the law of  $((\overline{q}^q_t)_{t\in[0,T]},Z)$ when $\gamma\rightarrow\infty$.
\end{theorem}

Using this convergence, we are then able to obtain the overdamped limit of the QSD.
\begin{theorem}[QSD overdamped limit]\label{cv etroite} Let Hypotheses~\ref{hyp F1 qsd} and~\ref{hyp O qsd} hold. The QSD $\mu^{(\gamma)}$ converges weakly, when $\gamma\rightarrow\infty$, to the probability measure $\mu^{(\infty)}$ on $D$ defined by:
 \begin{equation}\label{mesure limite qsd}
     \mu^{(\infty)}(\mathrm{d}q \mathrm{d}p):=\overline{\mu}(\mathrm{d}q) \frac{\mathrm{e}^{-\frac{\beta\vert p\vert^2}{2}}}{(2\pi\beta^{-1})^{\frac{d}{2}}} \mathrm{d}p. 
 \end{equation}
 Furthermore, the eigenvalue $\lambda_0^{(\gamma)}$ associated with the QSD satisfies
 $$\lambda_0^{(\gamma)}\underset{\gamma\rightarrow\infty}{\sim}\frac{\overline{\lambda}_0}{\gamma},$$
 where $\overline{\lambda}_0$ is defined in Theorem~\ref{qsd overdamped}.
\end{theorem}

Theorem~\ref{cv loi indep} is proven in Section~\ref{proof cv loi indep} and Theorem~\ref{cv etroite} is proven in Section~\ref{proof cv etroite}.

\section{Proofs}\label{tightness convergence langevin} 

We are interested  in the behavior of the QSD of the Langevin process defined in~\eqref{eq:Langevin qsd gamma} when $\gamma$ goes to infinity. 
We shall use the following notation: under Assumption~\ref{hyp F2 qsd}, for any $x=(q,p) \in \mathbb{R}^d$, we denote by $(X^{(\gamma),x}_t=(q^{(\gamma),x}_t,p^{(\gamma),x}_t))_{t \geq 0}$ the solution to~\eqref{eq:Langevin qsd gamma} with initial condition $x$, and by $(\overline{q}^{(\gamma),q}_t)_{t \geq 0}$ the solution to the stochastic differential equation~\eqref{eq:ovLangevin_intro} with initial condition $q$ and driven by the Brownian motion $(B^{(\gamma)}_t)_{t \geq 0}=(\frac{B_{\gamma t}}{\sqrt{\gamma}})_{t\geq0}$. All these processes are defined on the same probability space $(\Omega, \mathcal{F}, (\mathcal{F}_t)_{t \geq 0},\mathbb{P})$ and it is more convenient to keep track of the initial condition of each process with the superscript notation rather than in the probability measure. We also emphasize the fact that under Assumption~\ref{hyp F2 qsd}, uniqueness in distribution holds for the stochastic differential equation~\eqref{eq:ovLangevin_intro} and therefore the law of the process $(\overline{q}^{(\gamma),q}_t)_{t \geq 0}$ does not depend on $\gamma$.

\subsection{Proof of Theorem~\ref{cv loi indep}}\label{proof cv loi indep}
Let $x=(q,p)\in\mathbb{R}^{2d}$, $T>0$. First, let us briefly show the scheme of proof for the convergence of the marginal laws $(q^{(\gamma),x}_{\gamma t})_{t\in[0,T]}$ and $p^{(\gamma),x}_{\gamma T}$, which is standard in the litterature. Second, we introduce a perturbed Langevin dynamics having the same overdamped limit as the dynamics~\eqref{eq:Langevin qsd gamma}. The perturbed dynamics being an independent couple, we shall retrieve its overdamped limit through the overdamped limit of the marginals, from which we will conclude the proof of Theorem~\ref{cv loi indep}.  

Let us start by considering the convergence of the marginal laws of $((q^{(\gamma),x}_{\gamma t})_{t\in[0,T]},p^{(\gamma),x}_{\gamma T})$. Considering~\eqref{eq:Langevin qsd gamma}, we have almost surely, for $t\in[0,T]$,
\begin{equation}\label{eq:formintegqgammat}
  q^{(\gamma),x}_{\gamma t}=q-\frac{p^{(\gamma),x}_{\gamma t}-p}{\gamma}+\int_0^tF(q^{(\gamma),x}_{\gamma s}) \mathrm{d}s+\sqrt{2\beta^{-1}} B^{(\gamma)}_t.
\end{equation} 
Using Gronwall's lemma, we are able to deduce from this equality the inequalities~\eqref{borne 1} and~\eqref{borne 2} in Lemma~\ref{borne L1}, which ensure that the difference $(q^{(\gamma),x}_{\gamma t})_{t\in[0,T]}-(\overline{q}^{(\gamma),q}_t)_{t \in [0,T]}$ converges in probability to $0$, in the space of the bounded continuous functions on $[0,T]$. Furthermore, the process $(\overline{q}^{(\gamma),q}_t)_{t \in [0,T]}$ shares the same law as the process $(\overline{q}^q_t)_{t\in[0,T]}$, which does not depend on $\gamma$. Therefore, the law of the process $(q^{(\gamma),x}_{\gamma t})_{t\in[0,T]}$ converges weakly to the law of $(\overline{q}^q_t)_{t\in[0,T]}$ when $\gamma$ goes to infinity.

Moreover, it follows from~\eqref{eq:Langevin qsd gamma} that for all $t\geq0$,
\begin{equation}\label{eq:duhamel eq vitesse}
    p^{(\gamma),x}_{t}= p \mathrm{e}^{-\gamma t}+\mathrm{e}^{-\gamma  t}\int_0^{t}\mathrm{e}^{\gamma s} F(q^{(\gamma),x}_{s}) \mathrm{d}s+\sqrt{2\gamma\beta^{-1}}\mathrm{e}^{-\gamma t} \int_0^{t}\mathrm{e}^{\gamma s} \mathrm{d}B_s.
\end{equation}
For $t \geq 0$, let 
\begin{equation}\label{def Y_T gamma}
    Y_{t}^{(\gamma)}:=\sqrt{2\gamma\beta^{-1}}\mathrm{e}^{-\gamma^2 t} \int_0^{\gamma t}\mathrm{e}^{\gamma s} \mathrm{d}B_s,
\end{equation}
then evaluating~\eqref{eq:duhamel eq vitesse} at $t=\gamma T$ for $T\geq0$, we get
\begin{equation}\label{vitesse gamma T}
    p^{(\gamma),x}_{\gamma T}= p \mathrm{e}^{-\gamma^2 T}+\gamma \mathrm{e}^{-\gamma^2 T}\int_0^T\mathrm{e}^{\gamma^2 s} F(q^{(\gamma),x}_{\gamma s}) \mathrm{d}s+Y_T^{(\gamma)}. 
\end{equation}
Under Assumption~\ref{hyp F2 qsd}, $F$ is bounded. Besides, $Y_T^{(\gamma)} \sim \mathcal{N}_d(0,\beta^{-1}(1-\mathrm{e}^{-2 \gamma^2 T}) I_d)$. Therefore, $Y_T^{(\gamma)}\overset{\mathcal{L}}{\underset{\gamma\rightarrow\infty}{\longrightarrow}}Z$ where $Z\sim \mathcal{N}_{d}(0,\beta^{-1} I_d)$ and $p^{(\gamma),x}_{\gamma T}\overset{\mathcal{L}}{\underset{\gamma\rightarrow\infty}{\longrightarrow}}Z$ by Slutsky's theorem.  


The arguments above give the limit in law of the marginals of the couple $((q^{(\gamma),x}_{\gamma t})_{t\in[0,T]},p^{(\gamma),x}_{\gamma T})$. To prove Theorem~\ref{cv loi indep}, it remains to show that, in the limit $\gamma\rightarrow\infty$, the two random variables $(q^{(\gamma),x}_{\gamma t})_{t\in[0,T]}$ and $p^{(\gamma),x}_{\gamma T}$ are independent. This is done by introducing a perturbed Langevin process defined later. Let $h^{(\gamma)}_T:[0,T]\mapsto\mathbb{R}$ and the process $(Z_{t,T}^{(\gamma)})_{t\in[0,T]}$ be defined as follows:
\begin{equation}\label{expr h}
    \forall t\in[0,T],\qquad h^{(\gamma)}_T(t):=\frac{2}{\gamma}\frac{\mathrm{e}^{- \gamma^2 (T-t)}-\mathrm{e}^{- \gamma^2 T}}{1-\mathrm{e}^{-2 \gamma^2 T}}, 
\end{equation} 
$$Z_{t,T}^{(\gamma)}:=\sqrt{2\beta^{-1}} B^{(\gamma)}_t-h^{(\gamma)}_T(t) Y_T^{(\gamma)}.$$

\noindent Let $(\mathcal{F}^{(\gamma),Z}_t)_{t \in [0,T]}$ be the natural filtration of $(Z_{t,T}^{(\gamma)})_{t\in[0,T]}$. Under Assumption~\ref{hyp F2 qsd}, Itô's fixed point argument~\cite[Thm 2.9 p. 289]{Karatzas} shows that the stochastic differential equation 
\begin{equation}\label{Overdamed Langevin 2}
  \left\{
    \begin{array}{ll}
        \mathrm{d}w^{(\gamma),q}_t=F(w^{(\gamma),q}_t) \mathrm{dt}+\mathrm{d}Z_{t,T}^{(\gamma)} , \\
        w^{(\gamma),q}_0=q,
    \end{array}
\right.   
\end{equation}
possesses a unique strong solution $(w^{(\gamma),q}_t)_{t \in [0,T]}$, which is thus adapted to $(\mathcal{F}^{(\gamma),Z}_t)_{t \in [0,T]}$.

\noindent The process $((w^{(\gamma),q}_t)_{t\in[0,T]},Y_T^{(\gamma)})$ satisfies the following lemmata. \begin{lemma}[Independence]\label{lemma indep} Under Assumption~\ref{hyp F2 qsd}, for all $T>0$, the process $(w^{(\gamma),q}_t)_{t\in[0,T]}$ is independent of the random variable $Y_T^{(\gamma)}$. 
\end{lemma} 
\begin{proof}
Let $T>0$. Since $(w^{(\gamma),q}_t)_{t\in[0,T]}$ is $\mathcal{F}^{(\gamma),Z}_T$-measurable, it is sufficient to prove that the process $(Z_{t,T}^{(\gamma)})_{t\in[0,T]}$ is independent of $Y_T^{(\gamma)}$. It is clear that for any $t_1, \ldots, t_k \in [0,T]$, the vector $(Z_{t_1,T}^{(\gamma)}, \ldots, Z_{t_k,T}^{(\gamma)}, Y_T^{(\gamma)})$ is Gaussian, therefore the independence is satisfied if and only if for all $t\in[0,T]$, the covariance matrix of $(Z_{t,T}^{(\gamma)},Y_T^{(\gamma)})$ is null, which is indeed the case by an easy computation.
\end{proof}  

\begin{lemma}[Perturbed Langevin]\label{borne L1} Let Assumption~\ref{hyp F2 qsd} hold.
There exists $C>0$ such that for all $T>0$, $x=(q,p)\in\mathbb{R}^{2d}$, $\gamma>1$,
\begin{enumerate} 
    \item\label{borne 1} $\mathbb{E}\left[\sup_{t\in[0,T]}\left\vert q^{(\gamma),(q,p)}_{\gamma t}-w^{(\gamma),q}_t\right\vert\right]\leq \frac{C}{\gamma}\left(1+\vert p\vert+ \sqrt{\log(1+\gamma^2T)}\right)\mathrm{e}^{CT}$,
    \item\label{borne 2} $\mathbb{E}\left[\sup_{t\in[0,T]}\left\vert w^{(\gamma),q}_{t}-\overline{q}^{(\gamma),q}_t\right\vert\right]\leq\frac{C}{\gamma}\mathrm{e}^{CT}$.
\end{enumerate}  
\end{lemma}
The proof of Lemma~\ref{borne L1} is postponed to Section~\ref{Proofs}. These two lemmata now yield the following proof of Theorem~\ref{cv loi indep}.

\begin{proof}[Proof of Theorem~\ref{cv loi indep}]
Let $T>0$, $x=(q,p)\in\mathbb{R}^{2d}$. Let $\Phi$ be a bounded $k_\Phi$-Lipschitz continuous function on $\mathcal{C}([0,T],\mathbb{R}^d)$ equipped with the supremum norm on $[0,T]$ and let $g$ be a bounded $k_g$-Lipschitz continuous function on $\mathbb{R}^d$. Our goal is to prove the following convergence: 
\begin{equation}\label{cv indep qsd}
    \mathbb{E}\left[\Phi((q^{(\gamma),x}_{\gamma t})_{t\in[0,T]}) g(p^{(\gamma),x}_{\gamma T})\right]\underset{\gamma\rightarrow\infty}{\longrightarrow}\mathbb{E}\left[\Phi((\overline{q}^q_{t})_{t\in[0,T]})\right] \mathbb{E}\left[g(Z)\right],
\end{equation}
where, in the right-hand side, $(\overline{q}^q_{t})_{t\in[0,T]}$ refers to the solution of~\eqref{eq:ovLangevin_intro} (which we recall has the same law as all processes $(\overline{q}^{(\gamma),q}_{t})_{t\in[0,T]}$ for $\gamma>0$).

By \emph{\eqref{borne 1}} in Lemma~\ref{borne L1} and~\eqref{vitesse gamma T}, there exists $C'>0$, depending on $T$, such that for all $\gamma>1$, 
\begin{align*}
    &\left\vert\mathbb{E}\left[\Phi((q^{(\gamma),x}_{\gamma t})_{t\in[0,T]}) g(p^{(\gamma),x}_{\gamma T})\right]-\mathbb{E}\left[\Phi((w^{(\gamma),q}_{t})_{t\in[0,t]}) g(Y_T^{(\gamma)})\right]\right\vert\\
    &\leq k_\Phi \Vert g\Vert_\infty\frac{C'}{\gamma}\left(1+\vert p\vert+ \sqrt{\log(1+\gamma^2T)}\right)+ k_g \Vert\Phi\Vert_\infty\left(\vert p\vert \mathrm{e}^{-\gamma^2 T}+\frac{\Vert F\Vert_\infty}{\gamma}\right),
\end{align*} 
which converges to $0$ when $\gamma\rightarrow\infty$. Furthermore, by Lemma~\ref{lemma indep},
$$\mathbb{E}\left[\Phi((w^{(\gamma),q}_{t})_{t\in[0,T]}) g(Y_T^{(\gamma)})\right]=\mathbb{E}\left[\Phi((w^{(\gamma),q}_{t})_{t\in[0,T]})\right] \mathbb{E}\left[g(Y_T^{(\gamma)})\right] .$$
Since,
$Y_T^{(\gamma)} \sim \mathcal{N}_d(0,\beta^{-1}(1-\mathrm{e}^{-2 \gamma^2 T}) I_d)$ then $Y_T^{(\gamma)}\overset{\mathcal{L}}{\underset{\gamma\rightarrow\infty}{\longrightarrow}}Z$ with $Z \sim \mathcal{N}_d(0,\beta^{-1} I_d )$. As a result, $\mathbb{E}[g(Y_T^{(\gamma)})]\underset{\gamma\rightarrow\infty}{\longrightarrow}\mathbb{E}[g(Z)]$. Besides, using \emph{\eqref{borne 2}} in Lemma~\ref{borne L1}, one obtains that  $$\mathbb{E}\left[\Phi((w^{(\gamma),q}_{t})_{t\in[0,T]})\right]-\mathbb{E}\left[\Phi((\overline{q}^{(\gamma),q}_{t})_{t\in[0,T]})\right]\underset{\gamma\rightarrow\infty}{\longrightarrow}0.$$
Moreover, $\mathbb{E}[\Phi((\overline{q}^{(\gamma),q}_{t})_{t\in[0,T]})]=\mathbb{E}[\Phi((\overline{q}^q_{t})_{t\in[0,T]})]$, since $(\overline{q}^{(\gamma),q}_{t})_{t\in[0,T]}$ and $(\overline{q}^q_{t})_{t\in[0,T]}$ share the same law, which concludes the proof of~\eqref{cv indep qsd}.
\end{proof}
\subsection{Proof of Theorem~\ref{cv etroite}}\label{proof cv etroite}

We consider in this section the weak limit, when $\gamma\rightarrow\infty$, of the QSD $\mu^{(\gamma)}$ of the Langevin process on $D$. Furthermore, we only assume here that $F$ satisfies Hypothesis~\ref{hyp F1 qsd}. In fact, we consider here the QSD on $D$ of the process~\eqref{eq:Langevin qsd gamma} which only depends on the values of the process inside $D$, hence on the values of $F$ inside $\mathcal{O}$ by Friedman's uniqueness result~\cite[Theorem 5.2.1.]{F}. As a result, one can extend $F$ arbitrarily outside of $\mathcal{O}$ so that it satisfies Assumption~\ref{hyp F2 qsd}. The notation for the overdamped Langevin process and its QSD remains the same as in Theorem~\ref{qsd overdamped}.

The idea of the proof of Theorem~\ref{cv etroite} is the following. We pick an arbitrary sequence $(\gamma_n)_{n\geq1}$ of positive numbers going to infinity. We first prove that the sequence of probability measures $(\mu^{(\gamma_n)})_{n\geq1}$ is tight. Then using Prokhorov's theorem we obtain a convergent subsequence $(\mu^{(\gamma'_n)})_{n\geq1}$ to a probability measure $\mu'$ on $D$. It is then left to prove that such a $\mu'$ is necessarily $\mu^{(\infty)}$ (see
~\eqref{mesure limite qsd}), whatever the sequence $(\gamma'_n)_{n\geq1}$. As a result, $\mu^{(\gamma)}$ necessarily converges weakly, when $\gamma$ goes to infinity, to $\mu^{(\infty)}$ defined by~\eqref{mesure limite qsd}.

This approach allows us to obtain the existence of a weak limit for the QSD and to identify it. However, it does not provide a speed of convergence, which can be interesting in applied contexts for instance. Nonetheless, in the simpler case of a stationary distribution, we are able to obtain a speed of convergence in Wasserstein distance for the overdamped limit of the stationary distribution, using estimates from Lemma~\ref{borne L1} and Theorem~\ref{cv loi indep}, see~\cite{MonRam}. Obtaining a speed of convergence for the QSD instead is still an open problem which is being looked at.

Now, let $(\gamma_n)_{n\geq1}$ be an arbitrary sequence of positive numbers going to infinity. Let us first prove that the sequence  $(\mu^{(\gamma_n)})_{n\geq1}$ is tight. This is the consequence of the following lemma which is proven in Section~\ref{Proofs}.  

\begin{proposition}[Estimates on $\psi^{(\gamma)}$]\label{properties psi gamma} Under Hypotheses~\ref{hyp F1 qsd} and~\ref{hyp O qsd}, the density $\psi^{(\gamma)}$ of the QSD $\mu^{(\gamma)}$ of~\eqref{eq:Langevin qsd gamma} satisfies the following properties:
\begin{enumerate} 
    \item $ \limsup_{\gamma\rightarrow\infty}\big\Vert\psi^{(\gamma)}\big\Vert_\infty<\infty$, 
    \item $ \limsup_{\gamma\rightarrow\infty}\sup_{q\in\mathcal{O}}\int_{\mathbb{R}^d}\psi^{(\gamma)}(q,p) \mathrm{d}p<\infty$,
    \item $ \limsup_{\gamma\rightarrow\infty}\iint_{D}\vert p\vert \psi^{(\gamma)}(q,p) \mathrm{d}p \mathrm{d}q<\infty$.
\end{enumerate}   
\end{proposition}
  
\begin{corollary}[Tightness]\label{tightness} Under Hypotheses~\ref{hyp F1 qsd} and~\ref{hyp O qsd}, the sequence of probability measures $(\mu^{(\gamma_n)})_{n\geq1}$ is tight. 
\end{corollary}
\begin{proof}
Recall that for any $n\geq1$, $\mu^{(\gamma_n)}$ is supported in $D$. For $k\geq1$, let $K_k$ be the compact subset of $D$ defined by   $$K_k:=\left\{(q,p)\in D : \vert p\vert\leq k, \mathrm{d}_\partial(q)\geq\frac{1}{k} \right\},$$
where $\mathrm{d}_\partial$ is the Euclidean distance to the boundary $\partial\mathcal{O}$. Let $K_k^c:=D\setminus K_k=\{(q,p)\in D : \vert p\vert> k\}\cup\{(q,p)\in D : \mathrm{d}_\partial(q)<\frac{1}{k}\}$. Let us prove the following limit
\begin{equation}\label{tightness estimate}
    \lim_{k\rightarrow\infty}\limsup_{n\rightarrow\infty}\mu^{(\gamma_n)}(K_k^c)=0, 
\end{equation} which immediately yields the required tightness.  

Let $\mathcal{O}_{k}:=\{q\in\mathcal{O} : \mathrm{d}_\partial(q)<\frac{1}{k}\}$. For all $n\geq1$,
\begin{align*}
     \mu^{(\gamma_n)}(K_k^c) 
    &\leq\iint_{D\cap\{\vert p\vert>k\}}\psi^{(\gamma_n)}(q,p) \mathrm{d}p \mathrm{d}q+\iint_{D\cap\{\mathrm{d}_\partial(q)<\frac{1}{k}\}}\psi^{(\gamma_n)}(q,p) \mathrm{d}p \mathrm{d}q\\
    &\leq\iint_{D\cap\{\vert p\vert>k\}}\psi^{(\gamma_n)}(q,p) \frac{\vert p\vert}{k} \mathrm{d}p \mathrm{d}q+\int_{\mathrm{d}_\partial(q)<\frac{1}{k}}\left(\int_{\mathbb{R}^d}\psi^{(\gamma_n)}(q,p) \mathrm{d}p\right) \mathrm{d}q\\
    &\leq\frac{\iint_{D}\psi^{(\gamma_n)}(q,p) \vert p\vert \mathrm{d}p \mathrm{d}q}{k}+\vert\mathcal{O}_k\vert \sup_{q\in\mathcal{O}}\int_{\mathbb{R}^d}\psi^{(\gamma_n)}(q,p) \mathrm{d}p .
\end{align*}
The convergence~\eqref{tightness estimate} then follows from Proposition~\ref{properties psi gamma}, which concludes the proof.
\end{proof}
Last, we state and prove here the following lemma which is used later in the proof of Theorem~\ref{cv etroite}.
\begin{lemma}[Convergence in distribution]\label{lem:cv loi discont}
Let Assumptions~\ref{hyp F2 qsd} and~\ref{hyp O qsd} hold. Let $f\in\mathcal{C}^b(\mathcal{O})$, $g\in\mathcal{C}^b(\mathbb{R}^d)$. For all $(q,p)\in D$ and $t>0$,
\begin{equation}\label{eq:cont distrib disc}
    \mathbb{E}\left[f(q^{(\gamma),(q,p)}_{\gamma t}) g(p^{(\gamma),(q,p)}_{\gamma t}) \mathbb{1}_{\tau^{(\gamma),(q,p)}_\partial>\gamma t}\right]\underset{\gamma\rightarrow\infty}{\longrightarrow}\mathbb{E}\left[f(\overline{q}^q_t) \mathbb{1}_{\overline{\tau}^{q}_\partial>t}\right] \mathbb{E}\left[g(Z)\right].
\end{equation}
\end{lemma}
\begin{proof}Let $(q,p)\in D$ and $T>0$. Since $\mathcal{O}$ is an open set, we have for any $\gamma>0$,
\begin{equation*}
  f(q^{(\gamma),(q,p)}_{\gamma T})g(p^{(\gamma),(q,p)}_{\gamma T}) \mathbb{1}_{\tau^{(\gamma),(q,p)}_\partial>\gamma T} = \Phi\left( (q^{(\gamma),(q,p)}_{\gamma t})_{t \in [0,T]}, p^{(\gamma),(q,p)}_{\gamma T}\right),
\end{equation*}
where $\Phi : \mathcal{C}([0,T],\mathbb{R}^d)\times\mathbb{R}^d \to \mathbb{R}$ is defined by
\begin{equation*}
  \Phi\left( (q_t)_{t \in [0,T]},z\right) = f(q_T)g(z)\mathbb{1}_{\inf_{t \in [0,T]} \mathrm{d}_\partial(q_t)>0},
\end{equation*}
and we take the convention that $\mathrm{d}_\partial(q')=0$ if $q' \not\in \mathcal{O}$. The functional $\Phi$ is not continuous on the space $\mathcal{C}([0,T],\mathbb{R}^d)\times\mathbb{R}^d$, which prevents us from applying Theorem~\ref{cv loi indep} directly. Indeed, take for example a continuous trajectory $(q_t)_{t\in[0,T]}$ on $[0,T]$ which hits the boundary $\partial\mathcal{O}$ and is reflected back into the domain $\mathcal{O}$. One can construct a sequence of functions $((q^{(n)}_t)_{t\in[0,T]})_{n\geq1}$ converging in the supremum norm to $(q_t)_{t\in[0,T]}$ such that for all $n\geq1$, $\inf_{t \in [0,T]} \mathrm{d}_\partial(q^{(n)}_t)>0$. As a result, $(q_t)_{t\in[0,T]}$ is an example of a discontinuity point of the function $\Phi$.

The discontinuity points of $\Phi$ are contained in the set of discontinuity points of $\mathbb{1}_{\inf_{t \in [0,T]} \mathrm{d}_\partial(q_t)>0}$, which can be characterized as follows. They correspond to the trajectories $(q_t)_{t\in [0,T]}$ which hit the boundary and remain on the boundary $\partial\mathcal{O}$ or come back inside $\mathcal{O}$. In fact if $(q_t)_{t\in [0,T]}$ is such that $\inf_{t\in[0,T]}\mathrm{d}_\partial(q_t)>0$ or $\sup_{t\in[0,T]}\mathrm{dist}(q_t,\mathbb{R}^d\setminus\overline{\mathcal{O}})>0$, then taking a sequence of functions $(q^{(n)}_t)_{t\in[0,T]}$ in $\mathcal{C}([0,T],\mathbb{R}^d)$ such that $\Vert q^{(n)}-q\Vert_\infty\leq\frac{\inf_{t\in[0,T]}\mathrm{d}_\partial(q_t)}{2}$ or $\Vert q^{(n)}-q\Vert_\infty\leq\frac{\sup_{t\in[0,T]}\mathrm{dist}(q_t,\mathbb{R}^d\setminus\overline{\mathcal{O}})}{2}$ then it follows from the $1-$Lipschitz continuity of the Euclidean distances $\mathrm{d}_\partial(\cdot)$ and $\mathrm{dist}(\cdot,\mathbb{R}^d\setminus\overline{\mathcal{O}})$ that 
\begin{equation*}
  \mathbb{1}_{\inf_{t \in [0,T]} \mathrm{d}_\partial(q^{(n)}_t)>0} \underset{n\rightarrow\infty}{\longrightarrow} \mathbb{1}_{\inf_{t \in [0,T]} \mathrm{d}_\partial(q_t)>0}.
\end{equation*}
As a consequence, the set of discontinuities of $\Phi$ is included in the set $S$ of continuous trajectories $(q_t)_{t \in [0,T]}$ such that there exists $t_\partial \in [0,T]$ for which $q_{t_\partial} \in \partial \mathcal{O}$ but for all $t \in [0,T]$, $q_t \in \overline{\mathcal{O}}$. Let us now justify that for all $q \in \mathcal{O}$, $\mathbb{P}((\overline{q}^q_t)_{t \in [0,T]}\in S)=0$. 
  Using the strong Markov property at $\overline{\tau}^q_\partial$, this is the case if for all $q\in\partial\mathcal{O}$, $\mathbb{P}(\overline{\tau}^q_\partial>0)=0$. This is clearly the case since for all $t>0$, $q\in\partial\mathcal{O}$, $\mathbb{P}(\overline{\tau}^q_\partial\leq t)=1$, see~\cite[p. 347]{FF}. Thus, the continuous mapping theorem ensures that $\Phi( (q^{(\gamma),(q,p)}_{\gamma t})_{t \in [0,T]}, p^{(\gamma),(q,p)}_{\gamma t})$ converges in distribution to $$\Phi\left( (\overline{q}^q_t)_{t \in [0,T]}, Z\right) = \mathbb{E}\left[f(\overline{q}^q_t) \mathbb{1}_{\overline{\tau}^{q}_\partial>t}\right] \mathbb{E}\left[g(Z)\right],$$ which completes the proof.
\end{proof}
 
\begin{proof}[Proof of Theorem~\ref{cv etroite}]

Notice that since the QSD $\mu^{(\gamma)}$ does not depend on the values of $F$ outside of~$\mathcal{O}$, we can consider here, up to a modification of $F$ outside of $\mathcal{O}$, that $F$ satisfies Assumption~\ref{hyp F2 qsd}. Therefore, the result of Theorem~\ref{cv loi indep} still applies in the current setting.  

By Corollary~\ref{tightness}, the sequence $(\mu^{(\gamma_n)})_{n\geq1}$ is tight, and therefore it is sequentially compact by Prokhorov's theorem. Let us consider a subsequence $(\gamma'_n)_{n\geq1}$ such that the sequence $(\mu^{(\gamma'_n)})_{n\geq1}$ converges weakly to a probability measure $\mu'$ on $D$ when $n$ goes to infinity. Let us now prove that $\mu'=\mu^{(\infty)}$ defined in~\eqref{mesure limite qsd} whatever the sequence $(\gamma'_n)_{n\geq1}$, which will conclude the proof. 

By Definition~\ref{def QSD 1} of a QSD, one easily deduce that for all $f\in\mathcal{C}^b(\mathcal{O})$, $g\in\mathcal{C}^b(\mathbb{R}^d)$ and all $t>0$,
\begin{align} 
&\iint_{D}\mu^{(\gamma'_n)}(\mathrm{d}q\mathrm{d}p) \mathbb{E}\left[f(q^{(\gamma'_n),(q,p)}_{\gamma'_n t}) g(p^{(\gamma'_n),(q,p)}_{\gamma'_n t}) \mathbb{1}_{\tau^{(\gamma'_n),(q,p)}_\partial>\gamma'_n t}\right]\nonumber\\
&=\mathrm{e}^{-\lambda_0^{(\gamma'_n)}\gamma'_n  t} \underbrace{\iint_D f(q) g(p) \mu^{(\gamma'_n)}(\mathrm{d}q\mathrm{d}p) }_{\underset{n\rightarrow \infty}{{\longrightarrow}}\iint_D f(q) g(p) \mu'(\mathrm{d}q\mathrm{d}p)},\label{ecriture qsd} 
\end{align}
where $\tau^{(\gamma'_n),(q,p)}_\partial$ denotes the exit time from $D$ for the process $(X^{(\gamma'_n),(q,p)}_t)_{t \geq 0}$.

Let $Z\sim\mathcal{N}_{d}(0,\beta^{-1} I_d )$ be a Gaussian vector independent of the process $(\overline{q}^q_t)_{t\in[0,T]}$ defined in~\eqref{eq:ovLangevin_intro}. Let us prove that the term in the left-hand side of the equality~\eqref{ecriture qsd} converges to $\iint_{D} \mathbb{E}[f(\overline{q}^q_t) \mathbb{1}_{\overline{\tau}^{q}_\partial>t}] \mathbb{E}[g(Z)]\mu'(\mathrm{d}q\mathrm{d}p)$. Considering the difference between the term in the left-hand side of the equality~\eqref{ecriture qsd} and  $\iint_{D} \mathbb{E}[f(\overline{q}^q_t) \mathbb{1}_{\overline{\tau}^{q}_\partial>t}] \mathbb{E}[g(Z)]\mu^{(\gamma'_n)}(\mathrm{d}q\mathrm{d}p)$ and partitioning the set $\{p\in\mathbb{R}^d\}$ into $\{\vert p\vert\leq K\}$ and $\{\vert p\vert> K\}$ for $K>0$, one obtains
\begin{align*}
    &\left\vert\iint_{D}\mu^{(\gamma'_n)}(\mathrm{d}q\mathrm{d}p) \left(\mathbb{E}\left[f(q^{(\gamma'_n),(q,p)}_{\gamma'_n t}) g(p^{(\gamma'_n),(q,p)}_{\gamma'_n t}) \mathbb{1}_{\tau^{(\gamma'_n),(q,p)}_\partial>\gamma'_n t}\right]-\mathbb{E}\left[f(\overline{q}^q_t) \mathbb{1}_{\overline{\tau}^{q}_\partial>t}\right] \mathbb{E}\left[g(Z)\right]\right)\right\vert\\
    &=\left\vert\iint_{D}\psi^{(\gamma'_n)}(q,p) \left(\mathbb{E}\left[f(q^{(\gamma'_n),(q,p)}_{\gamma'_n t}) g(p^{(\gamma'_n),(q,p)}_{\gamma'_n t}) \mathbb{1}_{\tau^{(\gamma'_n),(q,p)}_\partial>\gamma'_n t}\right]-\mathbb{E}\left[f(\overline{q}^q_t) \mathbb{1}_{\overline{\tau}^{q}_\partial>t}\right] \mathbb{E}\left[g(Z)\right]\right) \mathrm{d}p \mathrm{d}q\right\vert\\
    &\leq\big\Vert\psi^{(\gamma'_n)}\big\Vert_\infty\iint_{\mathcal{O}\times\{\vert p\vert\leq K\}}\left\vert\underbrace{\mathbb{E}\left[f(q^{(\gamma'_n),(q,p)}_{\gamma'_n t}) g(p^{(\gamma'_n),(q,p)}_{\gamma'_n t}) \mathbb{1}_{\tau^{(\gamma'_n),(q,p)}_\partial>\gamma'_n t}\right]-\mathbb{E}\left[f(\overline{q}^q_t) \mathbb{1}_{\overline{\tau}^{q}_\partial>t}\right] \mathbb{E}\left[g(Z)\right]}_{\underset{n\rightarrow\infty}{\longrightarrow}0\text{ by Lemma}~\ref{lem:cv loi discont}}\right\vert \mathrm{d}p \mathrm{d}q\\
    &+2 \Vert f\Vert_\infty \Vert g\Vert_\infty\iint_{\mathcal{O}\times\{\vert p\vert> K\}}\psi^{(\gamma'_n)}(q,p) \mathrm{d}p \mathrm{d}q .
\end{align*}  
Therefore, using Proposition~\ref{properties psi gamma} $(i)$ and the dominated convergence theorem to get that the limsup of the first term in the right-hand side is zero, 
\begin{align*}
    &\limsup_{n\rightarrow\infty}\left\vert\iint_{D}\mu^{(\gamma'_n)}(\mathrm{d}q\mathrm{d}p) \left(\mathbb{E}\left[f(q^{(\gamma'_n),(q,p)}_{\gamma'_n t}) g(p^{(\gamma'_n),(q,p)}_{\gamma'_n t}) \mathbb{1}_{\tau^{(\gamma'_n),(q,p)}_\partial>\gamma'_n t}\right]-\mathbb{E}\left[f(\overline{q}^q_t) \mathbb{1}_{\overline{\tau}^{q}_\partial>t}\right] \mathbb{E}\left[g(Z)\right]\right)\right\vert\\ 
    &\leq 2 \Vert f\Vert_\infty \Vert g\Vert_\infty\limsup_{n\rightarrow\infty}\iint_{\mathcal{O}\times\{\vert p\vert> K\}}\psi^{(\gamma'_n)}(q,p) \mathrm{d}p \mathrm{d}q\\
    &\leq 2 \Vert f\Vert_\infty \Vert g\Vert_\infty \limsup_{n\rightarrow\infty}\iint_{\mathcal{O}\times\{\vert p\vert> K\}}\psi^{(\gamma'_n)}(q,p) \frac{\vert p\vert}{K} \mathrm{d}p \mathrm{d}q\\
    &\leq \frac{2 \Vert f\Vert_\infty \Vert g\Vert_\infty}{K} \limsup_{n\rightarrow\infty}\iint_{D}\psi^{(\gamma'_n)}(q,p) \vert p\vert \mathrm{d}p \mathrm{d}q\underset{K\rightarrow\infty}{\longrightarrow}0,
\end{align*}
using Proposition~\ref{properties psi gamma} $(iii)$. 

Consequently,
$$\iint_{D}\mu^{(\gamma'_n)}(\mathrm{d}q\mathrm{d}p) \left(\mathbb{E}\left[f(q^{(\gamma'_n),(q,p)}_{\gamma'_n t}) g(p^{(\gamma'_n),(q,p)}_{\gamma'_n t}) \mathbb{1}_{\tau^{(\gamma'_n),(q,p)}_\partial>\gamma'_n t}\right]-\mathbb{E}\left[f(\overline{q}^q_t) \mathbb{1}_{\overline{\tau}^{q}_\partial>t}\right] \mathbb{E}\left[g(Z)\right]\right)\underset{n\rightarrow\infty}{\longrightarrow}0.$$
In addition,
\begin{align*}
     \iint_{D}\mu^{(\gamma'_n)}(\mathrm{d}q\mathrm{d}p) \mathbb{E}\left[f(\overline{q}^q_t) \mathbb{1}_{\overline{\tau}^{q}_\partial>t}\right] \mathbb{E}\left[g(Z)\right] 
    &=\mathbb{E}\left[g(Z)\right] \int_{\mathcal{O}}\mathbb{E}\left[f(\overline{q}^q_t) \mathbb{1}_{\overline{\tau}^{q}_\partial>t}\right] \left(\int_{p\in\mathbb{R}^d}\mu^{(\gamma'_n)}(\mathrm{d}q\mathrm{d}p)\right)\\
    &\underset{n\rightarrow\infty}{\longrightarrow}\mathbb{E}\left[g(Z)\right] \int_{\mathcal{O}}\mathbb{E}\left[f(\overline{q}^q_t) \mathbb{1}_{\overline{\tau}^{q}_\partial>t}\right] \left(\int_{p\in\mathbb{R}^d}\mu'(\mathrm{d}q \mathrm{d}p)\right), 
\end{align*}
since $q\in\mathcal{O}\mapsto\mathbb{E}\left[f(\overline{q}^q_t) \mathbb{1}_{\overline{\tau}^{q}_\partial>t}\right]$ is a bounded and continuous function on $\mathcal{O}$, see~\cite[Theorem 6.5.2]{F}. Consequently, taking $n\rightarrow\infty$ in the left-hand side of the equation~\eqref{ecriture qsd} and choosing $t=1$, it follows that $\lambda_0^{(\gamma'_n)}\gamma'_n$ converges to a value $\lambda'\in[0,\infty)$. Hence, taking $n\rightarrow\infty$ again in Equation~\eqref{ecriture qsd}, it follows that for all $t>0$,
\begin{equation}\label{eq2}
    \mathbb{E}\left[g(Z)\right] \int_{\mathcal{O}}\mathbb{E}\left[f(\overline{q}^q_t) \mathbb{1}_{\overline{\tau}^{q}_\partial>t}\right] \left(\int_{p\in\mathbb{R}^d}\mu'(\mathrm{d}q\mathrm{d}p)\right) =\mathrm{e}^{-\lambda't}\iint_Df(q)g(p)\mu'(\mathrm{d}q \mathrm{d}p).
\end{equation}
 Let $\mu'_\mathcal{O}$ be the probability measure on $\mathcal{O}$ defined by: $$\mu'_\mathcal{O}(\mathrm{d}q) = \int_{p \in \mathbb{R}^d} \mu' (\mathrm{d}q \mathrm{d}p).$$ Taking $g=1$ and $f=1$ in~\eqref{eq2}, we obtain that $\mathbb{P}_{\mu'_\mathcal{O}} (\overline{\tau}_\partial>t) = \exp(- \lambda' t)$. Since the equality can also be extended to all functions $f\in\mathrm{L}^\infty(\mathcal{O})$, using the density of $\mathcal{C}^b(\mathcal{O})$ in $\mathrm{L}^\infty(\mathcal{O})$, one gets for $g=1$ and $f=\mathbb{1}_{A}$ in~\eqref{eq2} with $A\in\mathcal{B}(\mathcal{O})$, 
$$\frac{\mathbb{P}_{\mu'_\mathcal{O}}(\overline{q}_t\in A, \overline{\tau}_\partial>t)}{\mathbb{P}_{\mu'_\mathcal{O}}(\overline{\tau}_\partial>t)}=\mu'_\mathcal{O}(A).$$ Therefore, $\mu'_\mathcal{O}$ is the unique QSD on $\mathcal{O}$ of $(\overline{q}_t)_{t\geq0}$ by Theorem~\ref{qsd overdamped}, which admits the density $\overline{\psi}$ with respect to the Lebesgue measure on $\mathcal{O}$. In particular, one has that $\lambda'=\overline{\lambda}_0$. Finally, reinjecting this equality into~\eqref{eq2}, we obtain that $\mu'$ satisfies the equality~\eqref{mesure limite qsd} since $Z\sim \mathcal{N}_d(0,\beta^{-1} I_d)$, which concludes the proof. 
\end{proof}
\subsection{Proofs of the technical results}\label{Proofs}
This section gathers the proofs of the technical results stated previously: Lemma~\ref{borne L1} and Proposition~\ref{properties psi gamma}.
\subsubsection{Proof of Lemma~\ref{borne L1}}
 
\begin{proof}[Proof of Lemma~\ref{borne L1}]
  
Let $T>0$, $x=(q,p)\in\mathbb{R}^{2d}$. Let us prove \emph{\eqref{borne 1}}. We recall from~\eqref{eq:formintegqgammat} that almost surely, for all $t\in[0,T]$, for all $\gamma>1$,
$$q^{(\gamma),x}_{\gamma t}=q-\frac{p^{(\gamma),x}_{\gamma t}-p}{\gamma}+\int_0^tF(q^{(\gamma),x}_{\gamma s}) \mathrm{d}s+\sqrt{2\beta^{-1}} B^{(\gamma)}_t .$$
Furthermore, by~\eqref{Overdamed Langevin 2}, almost surely, for all $t\in[0,T]$, 
$$w^{(\gamma),q}_{t}=q+\int_0^tF(w^{(\gamma),q}_{s}) \mathrm{d}s+Z_{t,T}^{(\gamma)},$$
where we recall $Z_{t,T}^{(\gamma)}=\sqrt{2\beta^{-1}} B^{(\gamma)}_t-h^{(\gamma)}_T(t) Y_T^{(\gamma)}$, with $Y_T^{(\gamma)}$ defined by~\eqref{def Y_T gamma}. It follows from~\eqref{expr h} that for all $T>0$, $\gamma>0$ and $t\in[0,T]$,
\begin{align*}
    h_T^{(\gamma)}(t) &\leq \frac{2}{\gamma} \frac{1-\mathrm{e}^{-\gamma^2T}}{1-\mathrm{e}^{-2\gamma^2T}}\\
    &\leq\frac{2}{\gamma}.
\end{align*}
Therefore, by Grönwall's Lemma, since $F$ is globally Lipschitz continuous with a Lipschitz coefficient $C_1>0$,
$$\sup_{t\in[0,T]}\left\vert q^{(\gamma),x}_{\gamma t}-w^{(\gamma),q}_{t}\right\vert\leq\left(\frac{\sup_{t\in[0,T]}\left\vert p^{(\gamma),x}_{\gamma t}-p\right\vert}{\gamma}+\frac{2}{\gamma} \left\vert Y_T^{(\gamma)}\right\vert\right) \mathrm{e}^{C_1 T} .$$
Moreover, by~\eqref{eq:duhamel eq vitesse} and~\eqref{def Y_T gamma}, almost surely, for $t\in[0,T]$,  
$$\frac{p^{(\gamma),x}_{\gamma t}-p}{\gamma}=-\frac{1-\mathrm{e}^{-\gamma^2 t}}{\gamma} p+\mathrm{e}^{-\gamma^2 t}\int_0^t\mathrm{e}^{\gamma^2 s} F(q^{(\gamma),x}_{\gamma s}) \mathrm{d}s+\frac{Y_t^{(\gamma)}}{\gamma}.$$
Therefore, since $\gamma>1$,
\begin{equation}\label{sup increment vitesse qsd}
    \mathbb{E}\left[\sup_{t\in[0,T]} \frac{\left\vert p^{(\gamma),x}_{\gamma t}-p\right\vert}{\gamma}\right]\leq \frac{\vert p\vert+\Vert F\Vert_\infty+\mathbb{E}\left[\sup_{t\in[0,T]} \vert Y_t^{(\gamma)}\vert\right]}{\gamma}.
\end{equation}
Let $(H^{(\gamma)}_t=((H^{(\gamma)}_t)_1,\dots,(H^{(\gamma)}_t)_d))_{t\in[0,T]}$ be the strong solution on $\mathbb{R}^d$ of the following Ornstein-Uhlenbeck SDE:
$$\mathrm{d}H^{(\gamma)}_t=-\gamma H^{(\gamma)}_t\mathrm{d}t+\mathrm{d}B_t, \qquad H^{(\gamma)}_0 = 0,$$
then it is easy to see that, almost surely, for $t\in[0,T]$, $Y^{(\gamma)}_t=\sqrt{2\gamma\beta^{-1}}H^{(\gamma)}_{\gamma t}$. Therefore, the Minkowski inequality applied to the Euclidean norm on $\mathbb{R}^d$ of $\vert Y_t^{(\gamma)}\vert$ ensures that
\begin{equation}\label{majoration OU}
    \sup_{t\in[0,T]} \vert Y_t^{(\gamma)}\vert\leq\sqrt{2\gamma\beta^{-1}}\sum_{i=1}^d\sup_{t\in[0,\gamma T]}\vert(H^{(\gamma)}_{t})_i\vert.
\end{equation}
A sharp inequality on the expectation in the summand above is provided in~\cite{IneqOU} and ensures the existence of a universal constant $C_2>0$ such that for all $t\in[0,T]$, $\gamma>0$ and $i\in\llbracket 1,d\rrbracket$,
$$\mathbb{E}\left[\sup_{t\in[0,\gamma T]}\vert(H^{(\gamma)}_{t})_i\vert\right]\leq \frac{C_2}{\sqrt{\gamma}}\sqrt{\log(1+\gamma^2T)}.$$
Reinjecting into~\eqref{majoration OU}, one gets $\mathbb{E}[\sup_{t\in[0,T]} \vert Y_t^{(\gamma)}\vert]\leq dC_2\sqrt{2\beta^{-1}}\sqrt{\log(1+\gamma^2T)}$. Therefore, the inequality~\eqref{sup increment vitesse qsd} ensures the existence of a constant $C_3>0$ such that for all $\gamma>1$, 
\begin{align*}
    \mathbb{E}\left[\sup_{t\in[0,T]} \frac{\left\vert p^{(\gamma),x}_{\gamma t}-p\right\vert}{\gamma}\right]\leq \frac{C_3}{\gamma}\left(1+\vert p\vert+ \sqrt{\log(1+\gamma^2T)}\right).
\end{align*}
Using the Cauchy-Schwarz inequality and the Itô isometry, one easily gets that $\mathbb{E}[\vert Y_T^{(\gamma)}\vert]\leq\sqrt{d\beta^{-1}}$. Therefore, for all $\gamma>1$, $T>0$, $t\in[0,T]$ and $x=(q,p)\in\mathbb{R}^{2d}$,
$$\mathbb{E}\left[\sup_{t\in[0,T]}\left\vert q^{(\gamma),x}_{\gamma t}-w^{(\gamma),q}_{t}\right\vert\right]\leq\frac{C_4}{\gamma}\left(1+\vert p\vert+ \sqrt{\log(1+\gamma^2T)}\right)\mathrm{e}^{C_1T}.$$ 
This concludes the proof of \emph{\eqref{borne 1}} and the proof of \emph{\eqref{borne 2}} also follows from the use of Gronwall's Lemma along with the previous estimates. 
\end{proof}
\subsubsection{Proof of Proposition~\ref{properties psi gamma}}
Let us now prove Proposition~\ref{properties psi gamma}. In order to do so, we resort to the two following results.
\begin{proposition}[Principal eigenvalue]\label{bornitude lambda} Under Hypotheses~\ref{hyp F1 qsd} and~\ref{hyp O qsd},
\begin{equation}\label{borne lambda}
    \limsup_{\gamma\rightarrow\infty} \lambda_0^{(\gamma)}\gamma<\infty.
\end{equation} 
\end{proposition}
The proof of Proposition~\ref{bornitude lambda} is postponed to the next section. 
In order to state the next lemma, let us first recall some results obtained in~\cite{LelRamRey} related to the transition density of the Langevin process~\eqref{eq:Langevin qsd gamma} absorbed at the boundary $\partial D$. 
 
 The transition kernel $\mathrm{P}^D_t$ of the process $(X_t)_{t\geq0}$ absorbed at the boundary $\partial D$ is defined by:
$$\forall t>0, \quad \forall x\in \overline{D}, \quad \forall A\in\mathcal{B}(D), \qquad \mathrm{P}^D_t(x,A):=\mathbb{P}_x(X_t\in A,\tau_\partial>t).$$
It has been shown in~\cite[Theorem 2.20]{LelRamRey} that $\mathrm{P}^D_t$ admits a smooth transition density 
$$(t,x,y)\in\mathbb{R}_+^*\times\overline{D}\times\overline{D}\mapsto\mathrm{p}^D_t(x,y)\in\mathcal{C}^\infty(\mathbb{R}_+^*\times D\times D)\cap\mathcal{C}(\mathbb{R}_+^*\times\overline{D}\times\overline{D}),$$ which admits the following Gaussian upper-bound, see~\cite[Theorem 2.19]{LelRamRey}.

\begin{theorem}[Gaussian upper-bound]\label{borne densite thm qsd}
Under Hypotheses~\ref{hyp F1 qsd} and~\ref{hyp O qsd}, the transition density $\mathrm{p}^D_t(x,y)$ is such that for all $\alpha\in (0,1)$, there exists $c_\alpha>0$, depending only on $\alpha$, such that for all $t>0$, for all $x,y\in D$,
\begin{equation} \label{borne densité qsd}
    \mathrm{p}^D_t(x,y)\leq\frac{1}{\alpha^d} \sum_{j=0}^\infty \frac{\left( \Vert F\Vert_\infty c_\alpha \sqrt{\pi t}  \right)^j}{(2\gamma\beta^{-1})^{j/2} \Gamma\left(\frac{j+1}{2}\right)} \widehat{\mathrm{p}}^{(\alpha)}_{t}(x,y),
\end{equation}
where $\Gamma$ is the Gamma function and
$\widehat{\mathrm{p}}^{(\alpha)}_{t}(x,y)$ is the transition density
of the Gaussian process
 $(\widehat{q}^{(\alpha)}_t, \widehat{p}^{(\alpha)}_t)_{t \geq 0}$ defined by
\begin{equation}\label{eq:processus alpha}
  \left\{\begin{aligned}
    \mathrm{d}\widehat{q}^{(\alpha)}_t &= \widehat{p}^{(\alpha)}_t \mathrm{d}t,\\
    \mathrm{d}\widehat{p}^{(\alpha)}_t &= -\gamma\widehat{p}^{(\alpha)}_t \mathrm{d}t + \frac{\sqrt{2\gamma\beta^{-1}}}{\sqrt{\alpha}}\mathrm{d}B_t.
  \end{aligned}\right.
\end{equation}  
\end{theorem}  
\begin{remark}\label{rk:gaussian ub gamma}
Notice that, in particular, for all $\alpha\in (0,1)$, there exists $c_\alpha>0$, depending only on $\alpha$, such that for all $t>0$, for all $x,y\in D$,
$$\mathrm{p}^D_{\gamma t}(x,y)\leq\frac{1}{\alpha^d} \sum_{j=0}^\infty \frac{\left( \Vert F\Vert_\infty c_\alpha \sqrt{\pi t}  \right)^j}{(2\beta^{-1})^{j/2} \Gamma\left(\frac{j+1}{2}\right)} \widehat{\mathrm{p}}^{(\alpha)}_{\gamma t}(x,y),$$
where the prefactor is now independent of $\gamma$.
\end{remark}
The purpose of the next lemma is to give some estimates satisfied by the transition density $\widehat{\mathrm{p}}^{(\alpha)}_t$ introduced in Theorem~\ref{borne densite thm qsd}, which will prove to be useful for the proof of Proposition~\ref{properties psi gamma}.

\noindent Let $\Phi_1,\Phi_2$ be the following positive continuous functions on $\mathbb{R}$:
\begin{equation}\label{Phi_1}
\Phi_1:\rho\in\mathbb{R}\mapsto\begin{cases}
  \frac{1-\mathrm{e}^{-\rho}}{\rho}&\text{if $\rho\neq0$,}\\ 
  1 & \text{if $\rho=0$,} 
\end{cases}
\end{equation}
\begin{equation}\label{Phi_2}
\Phi_2:\rho\in\mathbb{R}\mapsto\begin{cases}
  \frac{3}{2 \rho^3}\left[2\rho-3+4 \mathrm{e}^{-\rho}-\mathrm{e}^{-2\rho}\right] & \text{if $\rho\neq0$,}\\
  1 &\text{if $\rho=0$.}
\end{cases}
\end{equation}
 
 \noindent One can show, see~\cite[Section 5.1]{LelRamRey}, that for all $t \geq 0$ and $\alpha\in(0,1]$, the vector $(\widehat{q}_t^{(\alpha)},\widehat{p}_t^{(\alpha)})$ admits the following law under $\mathbb{P}_{(q,p)}$
\begin{equation}\label{loi gaussienne}
\begin{pmatrix}
\widehat{q}^{(\alpha)}_t  \\
\widehat{p}^{(\alpha)}_t
\end{pmatrix}
  \sim  \mathcal{N}_{2d}\left(
\begin{matrix} 
\begin{pmatrix}
m_q(t)  \\
m_p(t) 
\end{pmatrix},
\frac{C(t)}{\alpha}
\end{matrix} 
\right)  , 
\end{equation}
where the mean vector is $$m_q(t):=q+t p \Phi_1(\gamma t),\qquad m_p(t):=p \mathrm{e}^{-\gamma  t},$$
and the matrix $C(t)$ is defined by:
$$C(t):= \begin{pmatrix}
c_{qq}(t)  I_d & c_{qp}(t)  I_d \\
c_{qp}(t)  I_d & c_{pp}(t)  I_d 
\end{pmatrix}, 
$$
where $I_d$ is the identity matrix in $\mathbb{R}^{d\times d}$ and
\begin{equation}\label{coeff cov}
    c_{qq}(t):=\frac{\sigma^2 t^3}{3} \Phi_2(\gamma t),\qquad c_{qp}(t):=\frac{\sigma^2 t^2}{2} \Phi_1(\gamma t)^2,\qquad c_{pp}(t):=\sigma^2t \Phi_1(2\gamma t)  . 
\end{equation} 
The determinant of the covariance matrix $\frac{C(t)}{\alpha}$ is $\mathrm{det}(\frac{C(t)}{\alpha})=(\frac{\sigma^4 t^4}{12\alpha} \phi(\gamma t))^d$ where $\phi$ is the positive continuous function defined by 
\begin{equation}\label{expr phi}
\phi:\rho\in\mathbb{R}\mapsto 4 \Phi_2(\rho) \Phi_1(2 \rho)-3 \Phi_1(\rho)^4=\begin{cases}
    \frac{6 (1-\mathrm{e}^{-\rho})}{\rho^4}\left[-2+\rho+(2+\rho)\mathrm{e}^{-\rho}\right]&\text{if $\rho\neq0$,}\\
   1 &\text{if $\rho=0$.}
\end{cases}
\end{equation}
Let us now prove the following lemma. 
 \begin{lemma}[Properties of the transition densities]\label{densite prop} 
    For any $t>0$, $\alpha\in(0,1)$, there exist $C_t>0$ and $\gamma_t>1$ such that for all $\gamma\geq\gamma_t$ and $(q,p),(q',p')\in\mathbb{R}^{2d}$, 
    \begin{equation}\label{borne densite gamma1}
       \widehat{\mathrm{p}}^{(\alpha)}_{\gamma t}((q,p),(q',p'))\leq C_t,
    \end{equation}
    and
    \begin{equation}\label{borne densite gamma2}
        \sup_{q'\in\mathcal{O}}\int_{\mathbb{R}^d}\widehat{\mathrm{p}}^{(\alpha)}_{\gamma t}((q,p),(q',p')) \mathrm{d}p'\leq C_t. 
    \end{equation} 
\end{lemma}
\begin{proof}[Proof of Lemma~\ref{densite prop}]
Let $t>0$ and $\alpha\in(0,1)$. The law of the Gaussian vector $(\widehat{q}_{\gamma t}^{(\alpha)},\widehat{p}_{\gamma t}^{(\alpha)})$ detailed above ensures that for all $(q,p),(q',p')\in\mathbb{R}^{2d}$,
$$\mathrm{p}^{(\alpha)}_{\gamma t}((q,p),(q',p'))\leq\frac{1}{ \sqrt{(2 \pi)^{2d} \mathrm{det}\left(\frac{C(\gamma t)}{\alpha}\right)}}=\frac{1}{ \sqrt{(2 \pi)^{2d} \left(\frac{(2\gamma\beta^{-1})^2 (\gamma t)^4}{12\alpha} \phi(\gamma^2 t)\right)^d}}.$$
Besides, since $\gamma^6 t^4 \phi(\gamma^2 t)\underset{\gamma\rightarrow\infty}{\longrightarrow}6 t$, the estimate~\eqref{borne densite gamma1} easily follows. 

Let us now prove~\eqref{borne densite gamma2}. Since
$\widehat{\mathrm{p}}^{(\alpha)}_{\gamma t}((q,p),(q',p'))$ is the density of the Gaussian vector $(\widehat{q}^{(\alpha)}_{\gamma t}, \widehat{p}^{(\alpha)}_{\gamma t})$, the
expression of
$\int_{\mathbb{R}^d}\widehat{\mathrm{p}}^{(\alpha)}_t((q,p),(q',p'))
\mathrm{d}p'$ corresponds to the marginal density of $\widehat{q}^{(\alpha)}_{\gamma t}$ under $\mathbb{P}_{(q,p)}$. Besides, under $\mathbb{P}_{(q,p)}$, 
\begin{equation*}
  \widehat{q}^{(\alpha)}_{\gamma t} \sim \mathcal{N}_d\left(q + \gamma tp\Phi_1(\gamma^2 t), \frac{c_{qq}(\gamma t)}{\alpha}I_d\right),\quad \frac{c_{qq}(\gamma t)}{\alpha} = \frac{2\beta^{-1} t^3}{3\alpha}\gamma^4\Phi_2(\gamma^2 t),
\end{equation*}
so that  
 \begin{align*}
   \int_{\mathbb{R}^d}\widehat{\mathrm{p}}^{(\alpha)}_{\gamma t}((q,p),(q',p')) \mathrm{d}p'&=\frac{(3\alpha)^{d/2}}{\left(4 \pi\beta^{-1}t^3 \gamma^4\Phi_2(\gamma^2 t)\right)^{d/2}} \mathrm{e}^{-\frac{3 \alpha}{4\beta^{-1}t^3 \gamma^4  \Phi_2(\gamma^2 t)}\left\vert q'-q-\gamma t p \Phi_1(\gamma^2 t) \right\vert^2} \\
    &\leq \frac{(3\alpha)^{d/2}}{\left(4 \pi\beta^{-1}t^3 \gamma^4\Phi_2(\gamma^2 t)\right)^{d/2}}. 
\end{align*}
Since $t^3\gamma^4\Phi_2(\gamma^2 t)\underset{\gamma\rightarrow\infty}{\longrightarrow}3 t$, the upper bound~\eqref{borne densite gamma2} immediately follows.
\end{proof}
Using the Gaussian upper-bound recalled in Remark~\ref{rk:gaussian ub gamma}, we are now able to prove Proposition~\ref{properties psi gamma}.
\begin{proof}[Proof of Proposition~\ref{properties psi gamma}]
For any $\alpha\in (0,1)$ and any $T>0$, there exists $C_{\alpha,T}>0$ such that for all $\gamma>0$, for all $t\in(0,T]$, for all $x,y\in D$,
\begin{equation}\label{Gaussian ub} 
    \mathrm{p}^D_{\gamma t}(x,y)\leq C_{\alpha,T} \widehat{\mathrm{p}}^{(\alpha)}_{\gamma t}(x,y),
\end{equation}
where $\widehat{\mathrm{p}}^{(\alpha)}_{s}(x,y)$ is the transition density
of the Gaussian process
 $(\widehat{q}^{(\alpha)}_s, \widehat{p}^{(\alpha)}_s)_{s \geq 0}$ defined in~\eqref{eq:processus alpha}.  

Let $\gamma>0$, by Definition~\ref{def QSD 1} of a QSD, $\mu^{(\gamma)}$ is such that for all $A\in\mathcal{B}(D)$,  
$$\mathbb{P}_{\mu^{(\gamma)}}(X^{(\gamma)}_\gamma\in A, \tau^{(\gamma)}_\partial>\gamma)=\mu^{(\gamma)}(A)\mathrm{e}^{-\lambda_0^{(\gamma)}\gamma},$$ since $\mathbb{P}_{\mu^{(\gamma)}}(\tau^{(\gamma)}_\partial>\gamma)=\mathrm{e}^{-\lambda_0^{(\gamma)}\gamma}$ because $\tau^{(\gamma)}_\partial$ follows the exponential law of parameter $\lambda_0^{(\gamma)}$, see~\cite[Theorem 2.13]{LelRamRey2}. 

The equality above being satisfied for any $A\in\mathcal{B}(D)$, and since $\mu^{(\gamma)}$ has the continuous density $\psi^{(\gamma)}$ with respect to the Lebesgue measure on $D$, one deduces that for all $(q',p')\in D$,
$$\psi^{(\gamma)}(q',p')=\mathrm{e}^{\lambda_0^{(\gamma)}\gamma} \iint_{D}\psi^{(\gamma)}(q,p) \mathrm{p}^D_{\gamma}((q,p),(q',p')) \mathrm{d}p \mathrm{d}q .$$  
Let $\alpha\in(0,1)$. Using Remark~\ref{rk:gaussian ub gamma}, there exists $C>0$ such that for all $\gamma>0$ and $(q',p')\in D$, 
\begin{equation}\label{ineq1 qsd}
    \psi^{(\gamma)}(q',p')\leq C \mathrm{e}^{\lambda_0^{(\gamma)}\gamma} \iint_{D}\psi^{(\gamma)}(q,p) \widehat{\mathrm{p}}^{(\alpha)}_{\gamma}((q,p),(q',p')) \mathrm{d}p \mathrm{d}q,  
\end{equation}
where $\widehat{\mathrm{p}}^{(\alpha)}_t$ is the transition density of the process $(\widehat{q}^{(\alpha)}_t, \widehat{p}^{(\alpha)}_t)_{t \geq 0}$ defined in~\eqref{eq:processus alpha}. By Proposition~\ref{bornitude lambda} and the upper-bounds~\eqref{borne densite gamma1} and~\eqref{borne densite gamma2} in Lemma~\ref{densite prop}, the first two estimates in Proposition~\ref{properties psi gamma} follow from~\eqref{ineq1 qsd} and the fact that $\psi^{(\gamma)}$ is the density of a probability measure on $D$. It remains now to prove the last estimate in Proposition~\ref{properties psi gamma}. 

It follows from Fubini-Tonelli's theorem and the inequality~\eqref{ineq1 qsd} that
\begin{equation}\label{maj2}
  \iint_{D}\psi^{(\gamma)}(q',p')\vert p'\vert \mathrm{d}p' \mathrm{d}q'\leq C \mathrm{e}^{\lambda_0^{(\gamma)}\gamma} \iint_{D}\psi^{(\gamma)}(q,p) \left(\iint_{D}\widehat{\mathrm{p}}^{(\alpha)}_{\gamma}((q,p),(q',p'))\vert p'\vert \mathrm{d}p' \mathrm{d}q'\right)\mathrm{d}p \mathrm{d}q.  
\end{equation}
Let us now prove that $$\limsup_{\gamma \to \infty} \sup_{(q,p)\in D} \iint_{D}\widehat{\mathrm{p}}^{(\alpha)}_{\gamma}((q,p),(q',p'))\vert p'\vert \mathrm{d}p' \mathrm{d}q' < \infty,$$ this will conclude the proof using~\eqref{borne lambda} and~\eqref{maj2}.

Let us start by rewriting, for any $(q,p) \in D$ and $\gamma>0$,
\begin{align*}
  \iint_{D}\widehat{\mathrm{p}}^{(\alpha)}_{\gamma}((q,p),(q',p'))\vert p'\vert \mathrm{d}p' \mathrm{d}q' &= \mathbb{E}_{(q,p)}\left[\mathbb{1}_{\widehat{q}^{(\alpha)}_\gamma \in \mathcal{O}}|\widehat{p}^{(\alpha)}_\gamma|\right] \\
  &\leq \mathbb{E}_{(q,p)}\left[|\widehat{p}^{(\alpha)}_\gamma-p\mathrm{e}^{-\gamma^2}|\right] + |p|\mathrm{e}^{-\gamma^2}\mathbb{P}_{(q,p)}\left(\widehat{q}^{(\alpha)}_\gamma \in \mathcal{O}\right),
\end{align*}
and recall that under $\mathbb{P}_{(q,p)}$, $\widehat{q}^{(\alpha)}_\gamma$ and $\widehat{p}^{(\alpha)}_\gamma$ have marginal distributions
\begin{equation*}
  \widehat{q}^{(\alpha)}_\gamma \sim \mathcal{N}_d\left(q+\gamma p \Phi_1(\gamma^2), \frac{2\beta^{-1}\gamma^4}{3\alpha} \Phi_2(\gamma^2)I_d\right), \qquad \widehat{p}^{(\alpha)}_\gamma \sim \mathcal{N}_d\left(p\mathrm{e}^{-\gamma^2}, \frac{2\beta^{-1} \gamma^2 \Phi_1(2\gamma^2)}{\alpha}I_d\right).
\end{equation*}
As a consequence, we deduce from the Cauchy-Schwarz inequality that
\begin{equation*}
  \mathbb{E}_{(q,p)}\left[|\widehat{p}^{(\alpha)}_\gamma-p\mathrm{e}^{-\gamma^2}|\right] \leq \sqrt{\frac{2d\beta^{-1} \gamma^2 \Phi_1(2\gamma^2)}{\alpha}},
\end{equation*}
the right-hand side of which is uniform in $(q,p)$ and is bounded when $\gamma \to \infty$. On the other hand, let us define $\delta:=\sup_{q,q'\in\mathcal{O}}\vert q-q'\vert$ (which is finite since $\mathcal{O}$ is bounded) and note that
\begin{align*}
  \mathbb{P}_{(q,p)}\left(\widehat{q}^{(\alpha)}_\gamma \in \mathcal{O}\right) &\leq \mathbb{P}_{(q,p)}\left(|\widehat{q}^{(\alpha)}_\gamma-q| \leq \delta\right)\\
  &= \mathbb{P}\left(\left|\gamma p \Phi_1(\gamma^2) + \sqrt{\frac{2\beta^{-1}\gamma^4}{3\alpha} \Phi_2(\gamma^2)}Z\right| \leq \delta\right), 
\end{align*}
where $Z \sim \mathcal{N}_d(0,I_d)$. By the triangle, Markov and Cauchy-Schwarz inequalities, if $|p|\not=0$ then
\begin{align*}
  \mathbb{P}\left(\left|\gamma p \Phi_1(\gamma^2) + \sqrt{\frac{2\beta^{-1}\gamma^4}{3\alpha} \Phi_2(\gamma^2)}Z\right| \leq \delta\right) &\leq \mathbb{P}\left(\sqrt{\frac{2\beta^{-1}\gamma^4}{3\alpha} \Phi_2(\gamma^2)} |Z| + \delta \geq \gamma |p| \Phi_1(\gamma^2)\right)\\
  &\leq \frac{\sqrt{\frac{2d\beta^{-1}\gamma^4}{3\alpha} \Phi_2(\gamma^2)} + \delta}{\gamma |p| \Phi_1(\gamma^2)},
\end{align*}
so that
\begin{equation*}
  |p|\mathrm{e}^{-\gamma^2}\mathbb{P}_{(q,p)}\left(\widehat{q}^{(\alpha)}_\gamma \in \mathcal{O}\right) \leq \mathrm{e}^{-\gamma^2}\frac{\sqrt{\frac{2d\beta^{-1}\gamma^4}{3\alpha} \Phi_2(\gamma^2)} + \delta}{\gamma \Phi_1(\gamma^2)},
\end{equation*}
the right-hand side of which is uniform in $(q,p)$ and vanishes when $\gamma \to \infty$.
\end{proof}

\subsubsection{Proof of Proposition~\ref{bornitude lambda}}

Let us finally prove Proposition~\ref{bornitude lambda}. We will need the following intermediate lemma.
\begin{lemma}[Uniform velocity tightness]\label{tension vitesse lemme} Let Assumption~\ref{hyp F2 qsd} hold.
 For every $\epsilon>0$, there exists $M>0$ such that for all $\gamma\geq4$,
\begin{equation}\label{majoration}
    \sup_{(q,p)\in\mathcal{O}\times\mathrm{B}(0,M)}\mathbb{P}\left(p^{(\gamma),(q,p)}_\gamma\notin\mathrm{B}(0,M)\right)\leq\epsilon,
\end{equation} 
where $\mathrm{B}(0,M):=\{p\in\mathbb{R}^d:\vert p\vert<M\}$.
 \end{lemma}
\begin{proof}
 Let $\epsilon>0$. Let us take $M\geq \frac{2 \sqrt{d\beta^{-1}}}{\epsilon}+\Vert F\Vert_\infty$. By~\eqref{eq:Langevin qsd gamma}, for all $x=(q,p)\in\mathcal{O}\times\mathrm{B}(0,M)$ and $\gamma\geq4$ (so that $\frac{M}{\gamma^2}+\frac{M}{\gamma}\leq\frac{M}{2}$),  
\begin{align*}
    \left\vert p^{(\gamma),x}_{\gamma}\right\vert&=\left\vert p \mathrm{e}^{-\gamma^2}+\gamma \mathrm{e}^{-\gamma^2}\int_0^1\mathrm{e}^{\gamma^2 s} F(q^{(\gamma),x}_{\gamma s}) \mathrm{d}s+Y_1^{(\gamma)}\right\vert\\
    &\leq \frac{M}{\gamma^2} \underbrace{\gamma^2 \mathrm{e}^{-\gamma^2}}_{<1}+\frac{\Vert F\Vert_\infty}{\gamma}+\left\vert Y_1^{(\gamma)}\right\vert\\
    &< \frac{M}{2}+\left\vert Y_1^{(\gamma)}\right\vert
\end{align*}
since $M\geq\Vert F\Vert_\infty$ and $\gamma\geq4$. Besides, 
$$\mathbb{P}\left(\left\vert Y_1^{(\gamma)}\right\vert> M/2\right)\leq\frac{\mathbb{E}\left[\left\vert Y_1^{(\gamma)}\right\vert\right]}{M/2}\leq\frac{2  \sqrt{d\beta^{-1}}}{M}\leq\epsilon$$ by definition of $M$. Therefore, for all $(q,p)\in\mathcal{O}\times\mathrm{B}(0,M)$,
\begin{equation*}
  \mathbb{P}\left(p^{(\gamma),(q,p)}_\gamma\notin\mathrm{B}(0,M)\right)\leq\epsilon.\qedhere
\end{equation*}
\end{proof}
Let us now prove Proposition~\ref{bornitude lambda}.
\begin{proof}[Proof of Proposition~\ref{bornitude lambda}]
Let $q_0\in\mathcal{O}$. Let $r\in(0,1)$ such that $\mathrm{B}(q_0,2r)\subset\mathcal{O}$. For $q\in\mathbb{R}^d$, we define the following stopping time: $$\overline{\tau}^{(\gamma),q}_0=\inf \{t>0: \overline{q}^{(\gamma),q}_t\notin\mathrm{B}(q_0,3r/2)\}.$$ Let also $a:=\inf_{q\in\mathrm{B}(q_0,r)}\mathbb{P}(\overline{q}^{(\gamma),q}_1\in \mathrm{B}(q_0,r/2),\overline{\tau}^{(\gamma),q}_0>1)$. Notice that $a>0$ since it is well known that the function $q\in\mathrm{B}(q_0,r)\mapsto\mathbb{P}(\overline{q}^{(\gamma),q}_1\in \mathrm{B}(q_0,r/2),\overline{\tau}^{(\gamma),q}_0>1)$ is continuous and positive on the compact set $\overline{\mathrm{B}(q_0,r)}$. Besides, $a$ does not depend on $\gamma$ since the law of the process $(\overline{q}^{(\gamma),q}_t)_{t \geq 0}$ does not depend on $\gamma$. Let us take $\epsilon\in(0,\frac{a}{4})$ and $M>0$ such that~\eqref{majoration} in Lemma~\ref{tension vitesse lemme} is satisfied. 

\noindent\textbf{Step 1:} Let us prove that there exists $\gamma_1>1$ such that
\begin{equation}\label{minoration 2}
    c:=\inf_{\gamma\geq\gamma_1}\inf_{(q,p)\in\mathrm{B}(q_0,r)\times\mathrm{B}(0,M)}\mathbb{P}\left(X^{(\gamma),(q,p)}_{\gamma}\in \mathrm{B}(q_0,r)\times\mathrm{B}(0,M),\tau^{(\gamma),(q,p)}_{\partial}>\gamma\right)>0. 
\end{equation} For $(q,p)\in\mathrm{B}(q_0,r)\times\mathrm{B}(0,M)$, 
\begin{align}
 &\mathbb{P}\left(X^{(\gamma),(q,p)}_{\gamma}\in \mathrm{B}(q_0,r)\times\mathrm{B}(0,M),\tau^{(\gamma),(q,p)}_{\partial}>\gamma\right)\nonumber\\
 &\geq\mathbb{P}\left(X^{(\gamma),(q,p)}_{\gamma}\in \mathrm{B}(q_0,r)\times\mathrm{B}(0,M),\tau^{(\gamma),(q,p)}_{\partial}>\gamma,\sup_{t\in[0,1]}\left\vert q^{(\gamma),(q,p)}_{\gamma t}-\overline{q}^{(\gamma),q}_t\right\vert\leq r/2\right).\label{minoration}
\end{align}  
By \emph{\eqref{borne 1}} and \emph{\eqref{borne 2}} in Lemma~\ref{borne L1}, there exists $C_1>0$, depending on $M$, such that for all $\gamma>4$, 
\begin{equation}\label{markov 1}
    \sup_{(q,p)\in \mathrm{B}(q_0,r)\times\mathrm{B}(0,M)}\mathbb{E}\left[\sup_{t\in[0,1]}\left\vert q^{(\gamma),(q,p)}_{\gamma t}-\overline{q}^{(\gamma),q}_t\right\vert\right]\leq C_1\frac{1+\sqrt{\log(1+\gamma^2)}}{\gamma}. 
\end{equation} 
Moreover, by~\eqref{majoration} in Lemma~\ref{tension vitesse lemme}, 
\begin{align*}
    &\mathbb{P}\left(X^{(\gamma),(q,p)}_{\gamma}\in \mathrm{B}(q_0,r)\times\mathrm{B}(0,M),\tau^{(\gamma),(q,p)}_{\partial}>\gamma,\sup_{t\in[0,1]}\left\vert q^{(\gamma),(q,p)}_{\gamma t}-\overline{q}^{(\gamma),q}_t\right\vert\leq r/2\right)\\
    &\geq\mathbb{P}\left(\overline{q}^{(\gamma),q}_1\in \mathrm{B}(q_0,r/2),\overline{\tau}^{(\gamma),q}_0>1,\sup_{t\in[0,1]}\left\vert q^{(\gamma),(q,p)}_{\gamma t}-\overline{q}^{(\gamma),q}_t\right\vert\leq r/2\right)-\epsilon, 
\end{align*}
by definition of $\overline{\tau}^{(\gamma),q}_0$ and since $\mathrm{B}(q_0,2r)\subset\mathcal{O}$. Using~\eqref{markov 1} and the Markov inequality, it follows that for all $(q,p)\in \mathrm{B}(q_0,r)\times\mathrm{B}(0,M)$, 
\begin{align*}
    &\mathbb{P}\left(\overline{q}^{(\gamma),q}_1\in \mathrm{B}(q_0,r/2),\overline{\tau}^{(\gamma),q}_0>1,\sup_{t\in[0,1]}\left\vert q^{(\gamma),(q,p)}_{\gamma t}-\overline{q}^{(\gamma),q}_t\right\vert\leq r/2\right)\\
    &\geq \mathbb{P}\left(\overline{q}^{(\gamma),q}_1\in \mathrm{B}(q_0,r/2),\overline{\tau}^{(\gamma),q}_0>1\right) - \frac{2 C_1}{\gamma r}(1+\sqrt{\log(1+\gamma^2)})\\
    &\geq a-\frac{2 C_1}{\gamma r}(1+\sqrt{\log(1+\gamma^2)}).  
\end{align*}
As a result, by~\eqref{minoration} and by definition of $a$ and $\epsilon$, for all $\gamma>4$,
\begin{align*}
    &\inf_{(q,p)\in\mathrm{B}(q_0,r)\times\mathrm{B}(0,M)}\mathbb{P}\left(X^{(\gamma),(q,p)}_{\gamma}\in \mathrm{B}(q_0,r)\times\mathrm{B}(0,M),\tau^{(\gamma),(q,p)}_{\partial}>\gamma\right)\\
    &\geq a-\frac{2 C_1}{\gamma r}(1+\sqrt{\log(1+\gamma^2)})-\frac{a}{4}.
\end{align*}
Hence, there exists $\gamma_1>4$ such that for all $\gamma\geq\gamma_1$,
$$\inf_{(q,p)\in\mathrm{B}(q_0,r)\times\mathrm{B}(0,M)}\mathbb{P}\left(X^{(\gamma),(q,p)}_{\gamma}\in \mathrm{B}(q_0,r)\times\mathrm{B}(0,M),\tau^{(\gamma),(q,p)}_{\partial}>\gamma\right)\geq\frac{a}{2}.$$

\noindent \textbf{Step 2:} Now let us prove~\eqref{borne lambda}. By~\eqref{minoration 2}, for all $\gamma\geq\gamma_1$,
\begin{align*}
    &\mathrm{e}^{\lambda_0^{(\gamma)}\gamma} \iint_{\mathrm{B}(q_0,r)\times\mathrm{B}(0,M)}\psi^{(\gamma)}(q,p) \mathbb{P}(X^{(\gamma),(q,p)}_{\gamma}\in \mathrm{B}(q_0,r)\times\mathrm{B}(0,M),\tau^{(\gamma),(q,p)}_{\partial}>\gamma) \mathrm{d}q \mathrm{d}p\\
    &\geq c \mathrm{e}^{\lambda_0^{(\gamma)}\gamma}\iint_{\mathrm{B}(q_0,r)\times\mathrm{B}(0,M)}\psi^{(\gamma)}(q,p) \mathrm{d}q \mathrm{d}p .
\end{align*}
Since $\psi^{(\gamma)}$ is the density of the QSD of the Langevin process $(X^{(\gamma)}_t)_{t\geq0}$ then 
\begin{align*}
    &\mathrm{e}^{\lambda_0^{(\gamma)}\gamma} \iint_{\mathrm{B}(q_0,r)\times\mathrm{B}(0,M)}\psi^{(\gamma)}(q,p) \mathbb{P}(X^{(\gamma),(q,p)}_{\gamma}\in \mathrm{B}(q_0,r)\times\mathrm{B}(0,M),\tau^{(\gamma),(q,p)}_{\partial}>\gamma) \mathrm{d}q \mathrm{d}p\\
    &\leq\mathrm{e}^{\lambda_0^{(\gamma)}\gamma} \iint_D\psi^{(\gamma)}(q,p) \mathbb{P}(X^{(\gamma),(q,p)}_{\gamma}\in \mathrm{B}(q_0,r)\times\mathrm{B}(0,M),\tau^{(\gamma),(q,p)}_{\partial}>\gamma) \mathrm{d}q \mathrm{d}p\\
    &=\iint_{\mathrm{B}(q_0,r)\times\mathrm{B}(0,M)}\psi^{(\gamma)}(q,p) \mathrm{d}q \mathrm{d}p .
\end{align*}
Consequently, for $\gamma\geq\gamma_1$,
$$c \mathrm{e}^{\lambda_0^{(\gamma)}\gamma}\iint_{\mathrm{B}(q_0,r)\times\mathrm{B}(0,M)}\psi^{(\gamma)}(q,p) \mathrm{d}q \mathrm{d}p\leq\iint_{\mathrm{B}(q_0,r)\times\mathrm{B}(0,M)}\psi^{(\gamma)}(q,p) \mathrm{d}q \mathrm{d}p$$
which concludes the proof since $\iint_{\mathrm{B}(q_0,r)\times\mathrm{B}(0,M)}\psi^{(\gamma)}(q,p) \mathrm{d}q \mathrm{d}p>0$.
\end{proof}

\ACKNO{Mouad Ramil was supported by the Région Ile-de- France through a Ph.D. fellowship of the
Domaine d’Intérêt Majeur (DIM) Math Innov. This work also benefited from the support of the project ANR QuAMProcs (ANR-19-CE40-0010) from the French National Research Agency. The author would also like to thank Tony Lelièvre and Julien Reygner for fruitfull discussions throughout this work.} 
\end{document}